\documentclass[letterpaper, 10pt, conference]{ieeeconf}

\IEEEoverridecommandlockouts

\usepackage{tikz}

\usetikzlibrary{automata,positioning, arrows}
\usepackage{amsmath,amssymb,amsfonts,amsthm}
\usepackage{algorithmic}
\usepackage{graphicx}
\usepackage{textcomp}
\usepackage{stfloats}
\usepackage{hyperref}
\usepackage{booktabs}
\let\labelindent\relax
\usepackage{enumitem}  
\usepackage{microtype}
\usepackage{booktabs}
\usepackage{bigints}
\usepackage{mathrsfs}
\usepackage{mathtools}

\usepackage{standalone}

\usepackage{color,soul}

\usepackage{multirow}
\usepackage{upgreek}
\usepackage{accents}

\newcommand{\hd}{d_{\Delta}}
\newcommand{\hdj}{d_{\Delta_{j}}}

\newcommand{\dms}{d_{\mathrm{m}_{_{\sigma(k)}}}}
\newcommand{\dmj}{d_{\mathrm{m}_{_{j}}}}
\newcommand{\dMs}{d_{\mathrm{M}_{_{\sigma(k)}}}}
\newcommand{\dMj}{d_{\mathrm{M}_{_{j}}}}

\newcommand{\dmone}{d_{\mathrm{m}_{_{1}}}}
\newcommand{\dMone}{d_{\mathrm{M}_{_{1}}}}
\newcommand{\dmtwo}{d_{\mathrm{m}_{_{2}}}}
\newcommand{\dMtwo}{d_{\mathrm{M}_{_{2}}}}

\newcommand{\xp}{\hat{x}}

\newcommand{\reals}{\mathbb{R}}

\newcommand{\naturals}{\mathbb{N}}
\newcommand{\A}{\mathrm{A}}

\newcommand{\B}{\mathrm{B}}

\newcommand{\K}{\mathrm{K}}

\newcommand{\F}{\mathrm{F}}

\newcommand{\Y}{\mathrm{Y}}
\newcommand{\Q}{\mathrm{Q}}
\newcommand{\Lm}{\mathrm{L}}

\newcommand{\J}{\mathrm{J}}
\newcommand{\Jf}{\mathfrak{I}}
\newcommand{\Xm}{\mathrm{X}}
\newcommand{\R}{\mathrm{R}}
\newcommand{\V}{\mathrm{V}}
\newcommand{\M}{\mathrm{M}}
\newcommand{\U}{\mathrm{U}}
\newcommand{\W}{\mathrm{W}}
\newcommand{\Z}{\mathrm{Z}}
\newcommand{\Ss}{\mathrm{S}}

\newcommand{\I}{\mathrm{I}}

\newcommand{\Ad}{\mathrm{A_d}}
\newcommand{\An}{\mathrm{A_n}}
\newcommand{\Am}{\mathrm{A_{m}}}
\newcommand{\AM}{\mathrm{A_{M}}}

\newcommand{\Amj}{\mathrm{A_{m_{_{j}}}}}
\newcommand{\Amjbar}{\xbar{\Am}_{_j}}
\newcommand{\Amjbarone}{\xbar{\Am}_{_1}}
\newcommand{\Amjbartwo}{\xbar{\Am}_{_2}}

\newcommand{\AMj}{\mathrm{A_{M_{_{j}}}}}
\newcommand{\AMjbar}{\xbar{\AM}_{_j}}
\newcommand{\AMjbarone}{\xbar{\AM}_{_1}}
\newcommand{\AMjbartwo}{\xbar{\AM}_{_2}}
\newcommand{\Amone}{\mathrm{A_{m_{_{1}}}}}
\newcommand{\AMone}{\mathrm{A_{M_{_{1}}}}}
\newcommand{\Amtwo}{\mathrm{A_{m_{_{2}}}}}
\newcommand{\AMtwo}{\mathrm{A_{M_{_{2}}}}}

\newcommand{\Pp}{\mathrm{P}}
\newcommand{\hc}{& \hspace{0.2cm}}
\newcommand{\T}{\top}

\newcommand{\Pj}{\Pp_{{_j}}} 

\newcommand{\Qonej}{\Q_{1_{_j}}} 
\newcommand{\Qtwoj}{\Q_{2_{_j}}} 
\newcommand{\Qthree}{\Q_{3_{_1}}} 

\newcommand{\Zonej}{\Z_{1_{_j}}} 
\newcommand{\Ztwoj}{\Z_{2_{_j}}} 
\newcommand{\Zthree}{\Z_{3_{_1}}} 

\newcommand{\Pjb}{\xbar{\Pp}_{{_j}}} 

\newcommand{\Qonejb}{\xbar{\Q}_{1_{_j}}} 
\newcommand{\Qtwojb}{\xbar{\Q}_{2_{_j}}} 
\newcommand{\Qthreeb}{\xbar{\Q}_{3_{_1}}}

\newcommand{\Zonejb}{\xbar{\Z}_{1_{_j}}} 
\newcommand{\Ztwojb}{\xbar{\Z}_{2_{_j}}} 
\newcommand{\Zthreeb}{\xbar{\Z}_{3_{_1}}} 

\newtheorem{remark}{Remark}{}
\newtheorem{problem}{Problem}
\newtheorem{thm}{Theorem}
\newtheorem{cor}{Corollary}
\newtheorem{lem}{Lemma}

    \newcommand\undermat[2]{%
  \makebox[0pt][l]{$\smash{\overbrace{\phantom{%
    \begin{matrix}#2\end{matrix}}}^{\text{$#1$}}}$}#2}
    
\newcommand*\xbar[1]{%
   \hbox{%
     \vbox{%
       \hrule height 0.7pt 
       \kern0.35ex
       \hbox{%
         \kern-0.0em
         \ensuremath{#1}%
         \kern-0.0em
       }%
     }%
   }%
} 

\definecolor{Ao}{rgb}{0.0, 0.5, 0.0}

\definecolor{blue-violet}{rgb}{0.54, 0.17, 0.89}

\begin{document}

\title{Systems with both constant and time-varying delays: a switched systems approach and application to observer-controller co-design}
\author{T.~Alves Lima \and M.~Della Rossa \and F.~Gouaisbaut \and R.~Jungers \and S.~Tarbouriech
\thanks{This study was financed in part by the ANR project HANDY 18-CE40-0010 and the European Research Council (ERC) under the \emph{European Union's Horizon 2022 research and innovation programme} under grant agreement No 864017 - L2C.
(Corresponding author: T. Alves Lima.)}
\thanks{T. Alves Lima, M. Della Rossa, and R. Jungers are with ICTEAM Institute, Universit\'{e} Catholique de Louvain, 1348, Louvain-la-Neuve, Belgium {\{thiago.alveslima, matteo.dellarossa, raphael.jungers\}@uclouvain.be}. S. Tarbouriech and F. Gouaisbaut are with LAAS-CNRS, Université de Toulouse, CNRS, Toulouse, France {\{tarbour, fgouaisb\}@laas.fr}.}
}

\maketitle

\begin{abstract} In this paper, we study the application of switched systems stability criteria to derive delay-dependent conditions for systems affected by both a constant and a time-varying delay. The main novelty of our approach lies on the use of path-complete Lyapunov techniques along with the proposition of a new modified functional to obtain convex analysis conditions while avoiding the need of computing a dwell time for each mode in a switched system representation, as usual in the \textit{switched approach} for time-delay systems. Furthermore, we leverage the developed analysis to obtain LMIs for the closed-loop stabilization of systems with time-varying sensor delays by means of an observer-based compensator. A numerical example illustrates the proposed methods.  
\end{abstract}

\section{Introduction}
\label{intro}

Time delay appears in a wide variety of systems and is frequently caused by transport or losses of mass/information. For a review of analysis and stabilization techniques for time-delay systems, we refer to~\cite{Fridman_2014}. 

Many works dealing with systems subject to time-varying delays have opted for a \emph{switched} formulation, considering the value of the delay as a source of switching behavior; for an overview concerning switching systems analysis see~\cite{liberzon}. Applications of this modeling strategy among the delay systems literature has been explored in multiple contexts. In~\cite{HETEL2008}, the equivalence between the existence of multiple Lyapunov functions for a switched delay-augmented representation and the existence of general quadratic Lyapunov-Krasovskii functionals (LKFs) for the original time-varying delay system is demonstrated. Similar equivalence results were proven, in the continuous-time setting, in~\cite{HaiMas15}. 
Related connections between switched systems and delay systems Lyapunov-based stability conditions were studied more recently in~\cite{8793178,Vittorio2020}. \textcolor{black}{The relations between switched systems, systems with data-losses and delay systems is studied in~\cite{JunKun18}, in which controllability and observability conditions are proposed.}
In addition to these general results, the switched-representation has been considered, while avoiding the so-called delay state-augmentation, in~\cite{JIANG2010313}, for the continuous-time case, splitting the delay interval in multiple zones and then imposing a dwell time on each subsystem, following the approach of~\cite[Section 3]{liberzon}. \textcolor{black}{In~\cite{SUN20082902}, stability of discrete time-delay systems is approached splitting the delay interval in two zones, leading to a $2$-mode delay switched system representation, with one subsystem possibly unstable; the overall stability is then ensured imposing dwell time and persistence of switching assumptions. }
 
 \textcolor{black}{
In this paper we study stability of delay systems with \emph{both} a bounded \emph{time-varying} delay $d(k)\in [d_m,d_M]$ for $k\in \naturals$ and a \emph{nominal} constant delay $d_n\in [d_m,d_M]$. The motivation to study this class of systems, and in particular to consider a term depending on a constant delay $d_n$, comes from the case of an observer-based control closed loop, where the observer is designed by considering an estimate of the unknown time-varying delay in an attempt to observe the non-delayed state $x(k)$. The introduction of the constant delay $d_n$ can improve the features of the closed-loop system in comparison with direct feedback of the output $y(k)$ which is affected by the time-varying delay $d(k)$. Nonetheless, the LKF-based stability conditions available in the literature do not model well the interaction between the constant $d_n$ and the time-varying delay $d(k)$, which motivates the switched systems representation used in this paper.}

Differently from \cite{HETEL2008}, we do not represent the considered delay system as a delay-free one, but rather as a switched system composed by two delay subsystems where the value of the time-varying delay $d(k)$  \emph{with respect to} the constant delay $d_n$ is used to define the underlying switching signal. To study stability and stabilizability conditions, we make use of path-complete Lyapunov functions approach (introduced in~\cite{AhmJun14} for delay-free switched systems). In this framework, the structure of the sufficient Lyapunov conditions is given by an underlying flexible combinatorial structure, a \emph{path-complete graph}, which somehow encodes the switching signals the system will follow. These techniques have been proven to be less conservative than more classical multiple Lyapunov techniques for switched systems. For an overview see~\cite{AhmJun14,PhiEss16} and references therein. We then leverage the developed conditions for the stabilization problem of systems with output time-varying delays. With the aid of Finsler's Lemma and algebraic manipulations, the proposed conditions  for the design of stabilizing nominal-delay observers are rewritten in the form of linear matrix inequalities (LMIs).

Summarizing, although splitting the time-delay interval in sub-zones is a rather common idea in delay systems literature (\cite{JIANG2010313,SUN20082902}), some important specificities of our manuscript can be enlisted: i) We propose a modified LKF structure (see Section III) that allows obtaining feasible convex conditions for the stability analysis in the arbitrary switching case (no dwell-time nor delay-free representations are needed, as commonly done in the literature). ii) The relation between a constant delay $d_n$ and a time-varying one $d(k)$ is taken into account by means of the switching signal, iii) the extension of path-complete Lyapunov criteria to LKFs, iv) the extension of the developed theory to derive design conditions in the form of LMIs for the observer-based \emph{stabilization} of systems with unknown output time-varying delays.



\textbf{Notation.} Given $\Y \in \reals^{n \times m}$, $\Y^{\T} $ denotes its transpose. 
Given $\W = \W^{\T}$, $\Z = \Z^{\T} $ in $ \reals^{n \times n}$, $\W \succ \Z$ ($\W \succeq \Z$)  means that $\W-\Z$ is positive definite (positive semi-definite). With $\mathbb{S}_n^{+}$ we denote the set of symmetric positive definite matrices. $\I$ and $0$ denote identity and null matrices of appropriate dimensions, with their dimensions explicitly stated whenever relevant. The operator \textit{He}$\{\Y\}$ denotes \textit{He}$\{\Y\}=\Y+\Y^{\T}$. The $\star$ symbol denotes symmetric blocks in the expression of a matrix. For matrices $\W$ and $\Z$, \textit{diag}$(\W,\Z)$ corresponds to the block-diagonal matrix. 

\section{Theoretical preliminaries}\label{sec:main}
In this paper, we study stability of systems of the form 
\begin{align}
   \hspace{-0.3cm} \begin{cases}
              x(k+1)=\A x(k)\hspace{-0.05cm}+\An x(k-d_n)\hspace{-0.05cm}\\~~~~~~~~~~~~~+\hspace{-0.05cm}\Ad x(k-d(k)),~\forall k\geq0\\
          x(k) = \phi(k), ~\forall k \in \left[-d_M,0 \right] 
          \end{cases}\label{clsystem}
          \end{align}
\noindent where $\phi(k)$ is the initial condition at the interval $\left[-d_M,0 \right]$, $x(k) \in \reals ^{n}$ is the plant state vector, $d(k)$ is a time-varying delay $1 \leq d_m \leq d(k) \leq d_M$ with known lower ($d_m$) and upper ($d_M$) limits, whereas the value of $d(k)$ at each sampling time $k$ is unknown. Furthermore, $d_n \in [d_m,d_M]$ is a constant delay within the same bounds.
As a preliminary step in the analysis of~\eqref{clsystem}, we study the stability for systems without constant delay (i.e. the case $\An=0$ in~\eqref{clsystem}). This allows us to present, for a simpler system, the tools and techniques which will be generalized, in Section~\ref{sec:main}, to the setting of system~\eqref{clsystem}.
\subsection{Delay-dependent stability of time-delay systems}
Consider the time-delay system
\begin{align}
  \begin{cases}
              x(k+1)=\A x(k) + \Ad x(k-d(k)),~\forall k\geq0\\
          x(k) = \phi(k),~\forall k \in \left[-d_M,0 \right] 
          \end{cases}\label{classicaldelaysystem}
          \end{align}
\noindent where $\phi(k)$ is the initial condition at the interval $\left[-d_M,0 \right]$, $x(k) \in \reals ^{n}$ is the plant state vector, and $d(k)$ is a time-varying delay $1 \leq d_m \leq d(k) \leq d_M$. When studying the stability of \eqref{classicaldelaysystem}, a common approach is to search for a Lyapunov-Krasovskii functional~\cite{Fridman_2014, Gu2003}. More precisely, we consider $V: \underbrace{\reals ^n \times \dotsc \times \reals ^n}_{d_M+1 \text{ times}} \rightarrow \reals ^+$,~and,~using~the~convention
\begin{equation}\label{eq:LKfunctional}
\begin{split}
    V(k) &\coloneqq
V(\xbar{x}(k)), \\
\xbar{x}(k) &\coloneqq \left [x^\top(k) \;x^\top(k-1)\;\dots\,\; x^\top(k-d_M)\right]^\top,
\end{split}
\end{equation}
we require that $V(k) > 0$ for all $\xbar{x}(k) \neq 0$ (and we say that $V$ is \emph{positive definite}), and the forward difference of $V$ with respect to~\eqref{classicaldelaysystem} is negative, i.e., $\Delta V(k) \coloneqq V(k+1)-V(k) < 0$ along the trajectories of~\eqref{classicaldelaysystem}. In this case, we say that $V$ is a \emph{Lyapunov Krasovskii functional (LKF)} for system~\eqref{classicaldelaysystem}. In order to obtain stability conditions in the form of LMIs, many LKF structures have been proposed in the literature, see \cite{Fridman_2014,Gu2003} and references therein. The manipulation of summation inequalities in the LKFs commonly lead to bounds of the form $\Delta V(k) \leq \xi^{\T}(k){\Phi}(d(k),d_m,d_M)\xi(k)$, where ${\Phi}(d(k),d_m,d_M)$ is a matrix-valued function and $\xi(k)$ is a vector function depending explicitly on $x(k+1)$, $x(k)$, $x(k-d(k))$, and also on the states delayed by the minimum and maximum delays, i.e., $x(k-d_m)$ and $x(k-d_M)$ (see, for example,~\cite{Shao_2011,Kwon_2013,Seuret_2015,Hien_2016}). LMI conditions are then obtained by either replacing $x(k+1)$ with the equation in~\eqref{classicaldelaysystem}  or by applying Finsler's Lemma (see Subsection \ref{standardLMIs}). Since standard manipulations already lead to bounds which depend on $x(k-d_m)$ and $x(k-d_M)$, the exact same LKF and manipulation procedure used to obtain stability conditions for~\eqref{classicaldelaysystem} can be applied for systems of the more general form 
\begin{equation} \label{systemswithupperbounds}
\begin{cases}
\!\begin{aligned}
          x(k&+1)=\A x(k) + \Am x(k-d_m)\\ &+ \AM x(k-d_M) + \Ad x(k-d(k)), ~\forall k\geq0
           \end{aligned}\\
          x(k) = \phi(k), ~\forall k \in \left[-d_M,0 \right].
          \end{cases}
\end{equation}

\begin{remark}
Although at first glance system \eqref{systemswithupperbounds} may seem more complex than system \eqref{clsystem}, its stability analysis in terms of convex conditions derived from LKFs is actually simpler, as one can directly employ the same traditional conditions used for \eqref{classicaldelaysystem}, which is not true for \eqref{clsystem}. In the case of \eqref{clsystem}, more complex LKFs involving more summation terms taking into account the delay $d_n$ that can be different from the maximum and minimum delay need to be applied. In this paper, we plan to leverage the same simpler conditions already stablished to \eqref{classicaldelaysystem} and \eqref{systemswithupperbounds} by using a switched representation and small changes to the structure of the LKF. This will be clear in the next sections.  
\end{remark}

\subsection{Standard analysis conditions with Finsler's Lemma}\label{standardLMIs}
\textcolor{black}{In this subsection we derive LMIs conditions, via a Lyapunov Krasovskii construction,
for stability of~\eqref{systemswithupperbounds}.
Since it represents a crucial tool in our algebraic manipulation, we recall here the celebrated Finsler's Lemma,~\cite{Mauricio_2001}.
}
\begin{lem}\label{lemma:Finsler} \cite{Mauricio_2001} Consider ${\Phi}={\Phi}^{\T} \in \reals^{n_{\xi} \times n_{\xi}}$, and $\Gamma \in \reals^{m_{\xi} \times n_{\xi}}$. The following statements are equivalent:
\begin{enumerate}[label=(\roman*)]
    \item  $\xi^{\T} {\Phi} \xi<0$, $\;\forall\xi \neq 0$ such that~$\Gamma \xi = 0$. \label{lemma:Finsler:item1}
    \item $\exists \Jf \in \reals^{n_{\xi} \times m_{\xi}}:$ ${\Phi}+\Jf \Gamma+\Gamma^{\T}\Jf^{\T} \prec 0$. \label{lemma:Finsler:item2}
    \item ${\Gamma^{\perp}}^{\T}{\Phi}{\Gamma^{\perp}} \prec 0$, where $\Gamma{\Gamma^{\perp}}=0$. \label{lemma:Finsler:item3}
\end{enumerate}
\end{lem}
Next, we review analysis conditions for systems~\eqref{systemswithupperbounds}.
We define $\hd:=d_M-d_m$, and the function $\gamma:\naturals\to \reals$ given~by
\begin{equation}\label{eq:gammafunc}
    \begin{cases}
            \gamma(d)=1, \text{ if } d=1, \\
            \gamma(d)=(d+1)/(d-1), \text{ if } d>1.
    \end{cases}
\end{equation}
\noindent Consider then the LKF-structure, inspired by \cite{Seuret_2015}, given by
\begin{equation}\label{LKFFunc}
    V(k) = V_{a}(k)+V_{b}(k)+V_{c}(k),
\end{equation}
\noindent where
\begin{equation}\label{eq:Lyapunovstrcture}
\begin{split}
&V_{a}(k) = w^{\T}(k)\Pp w(k), \\
&V_{b}(k) = \sum_{l=k-d_m}^{k-1} x^{\T}(l)\Q_1 x(l)+\sum_{l=k-d_M}^{k-d_m-1} x^{\T}(l)\Q_2x(l), \\
&\!\begin{aligned}
V_{c}(k) = & d_m  \sum_{l=-d_m+1}^{0} \sum_{i=k+l}^{k} \eta^{\T}(i)\Z_1\eta(i) \\ &+  \hd \sum_{l=-d_M+1}^{-d_m} \sum_{i=k+l}^{k} \eta^{\T}(i)\Z_2\eta(i),
\end{aligned}
\end{split}
\end{equation}
\noindent with $w(k):= \begin{bmatrix} x^{\T}(k) & \hspace{-0.15cm}\sum_{l=k-d_m}^{k-1} \hspace{-0.1cm}x^{\T}(l) & \hspace{-0.15cm}\sum_{l=k-d_M}^{k-d_m-1} x^{\T}(l) \end{bmatrix}^{\T}$ and $\eta(i)=x(i)-x(i-1)$.  
Supposing that ${\Pp} \in\mathbb{S}_{3n}^{+}$ and ${\Q}_1$, ${\Q}_2$, ${\Z}_1$, ${\Z}_2\in\mathbb{S}_{n}^{+}$ implies that $V$ is positive definite. 
\textcolor{black}{In what follows, we present a lemma allowing to perform stability analysis of~\eqref{systemswithupperbounds} using the Lyapunov-Krasovskii structure introduced in~\eqref{eq:Lyapunovstrcture}. The proposed LMI conditions are equivalent in conservatism to the ones in~\cite[Theorem 5]{Seuret_2015}, but are slightly different due to the application of Lemma~\ref{lemma:Finsler}. In~\cite{Seuret_2015}, instead, the stability conditions are obtained by direct substitution of the dynamics $x(k+1)$ in the manipulation of $\Delta V(k)$. The choice of using Lemma~\ref{lemma:Finsler} will be justified in Section~\ref{sec:observerdesign}. Some steps of the proof of the lemma below, especially involving the bounding of the term $\Delta V_c(k)$, are not made completely explicit, in order to avoid repetition of the manipulations already presented in~\cite{Seuret_2015}, to which we refer for the details. }

\begin{lem}\label{lem:stab}
Assume that there exist matrices ${\Pp} \in\mathbb{S}_{3n}^{+}$, ${\Q}_1$, ${\Q}_2$, ${\Z}_1$, ${\Z}_2\in\mathbb{S}_{n}^{+}$, ${\Xm}\in\reals^{2n\times2n}$ such that 
\begin{equation}\label{main_starnd}
{\Psi}_z \succ 0,
\!\begin{aligned}
~~{\Gamma^{\perp}}^{\T}{\Phi}(d_m){\Gamma^{\perp}} \prec 0, ~~{\Gamma^{\perp}}^{\T}{\Phi}(d_M){\Gamma^{\perp}} \prec 0,
\end{aligned}
\end{equation}
\noindent  hold with 
$
    {\Gamma^{\perp}} =
  \begin{bmatrix} 
  \A & \A_m & \Ad &  \A_M & 0_{n \times 3n} \\
  & & \I_{7n} &
   \end{bmatrix}$,
\begin{align*}
 &{\Phi}(d) = \Phi_1(d)+\mathcal{Q}
   + \Phi_3(d_m,d_M), \\
   &\Phi_1(d)=\W_{2}^{\T}(d_m,d_M){\Pp}\W_{2}(d_m,d_M)\\  
   &~~~~~~~~-\W_{1}^{\T}(d_m,d_M){\Pp}\W_{1}(d_m,d_M) \\ 
   &~~~~~~~~+\text{He}\left\{\W^{\T}(d){\Pp}\left(\W_2(d_m,d_M)-\W_1(d_m,d_M)\right)\right\}, \\
   &\Phi_3(d_m,d_M)=\W_3^{\T}(d_m^2 {{\Z}}_1+\hd^2 {{\Z}}_2) \W_3\\
   &~~~~~~~~-\W_{s}^{\T} \mathscr{Z}_1(d_m)\W_{s}-\W_{\Psi}^{\T}{\Psi}_z\W_{\Psi},
\end{align*}
\begin{equation*}
    {\Psi}_z = \begin{bmatrix}
    ~\mathscr{Z}_2 & {\Xm}~ \\
    ~\star & \mathscr{Z}_2~
    \end{bmatrix},
\end{equation*}
\begin{equation*}
    \mathscr{Z}_1(d_m) = \text{diag}\left({\Z}_1,3\gamma(d_m){\Z}_1\right), \mathscr{Z}_2 = \text{diag}\left({\Z}_2,3{\Z}_2\right),
\end{equation*}
$ \mathcal{Q} = \text{diag}(0,{\Q}_1,{\Q}_2-{\Q}_1,0,-{\Q}_2,0,0,0)$, $\gamma$ defined in \eqref{eq:gammafunc} and $\W_{\Psi},\M,\W_{s},\W_{3},\W_{1}(d_m,d_M)$, $ \W_{2}(d_m,d_M), \W(d)$ are given in Appendix~\ref{FirstAppendix}. Then, the LKF $\V$ defined in~\eqref{LKFFunc}-\eqref{eq:Lyapunovstrcture} is positive definite and satisfies $\V(k+1)-\V(k)<0$, for all $\xbar{x}(k)\neq 0$. In particular, system~\eqref{systemswithupperbounds} is asymptotically stable  for any time-varying delay $d_m \leq d(k) \leq d_M$.
\end{lem}
\begin{proof}
 Consider the LKF \eqref{LKFFunc}-\eqref{eq:Lyapunovstrcture} and the augmented vector
\begin{equation*}
    \begin{aligned}
\xi(k) &:=
\left[\begin{matrix}
  x(k+1)^{\T} & x(k)^{\T} & x(k-d_m)^{\T}
\end{matrix}\right.\\
&\qquad\qquad
\left.\begin{matrix}
  {}x(k-d(k))^{\T} & x(k-d_M)^{\T} & v^{\T}
\end{matrix}\right]^{\T},
\end{aligned}
\end{equation*}
\noindent with $v = \begin{bmatrix}
v_1^{\T} & v_2^{\T} & v_3^{\T}  
\end{bmatrix}^{\T}$ given by $ v_1 = \frac{1}{d_m+1}\sum_{l=k-d_m}^{k}x(l)$, $v_2 = \frac{1}{d(k)-d_m+1}\sum_{l=k-d(k)}^{k-d_m}x(l)$, $v_3= \frac{1}{d_M-d(k)+1}\sum_{l=k-d_M}^{k-d(k)}x(l)$.
\noindent By evaluating the forward difference of $V_a(k)$ and $V_b(k)$, we obtain
\begin{subequations}
\begin{align}
\begin{split}
    \Delta V_{a}(k) = \xi^{\T}(k)  \Phi_1(d(k)) \xi(k),
\end{split}\label{eqP}\\
\begin{split}
     \Delta V_{b}(k) = \xi^{\T}(k)\mathcal{Q} \xi(k).
\end{split}\label{eqQ}\\
\intertext{Using a summation version of the Wirtinger's integral inequality from \cite{SEURET2013} and the reciprocally convex Lemma \cite{PARK2011}, in the proof of \cite[Theorem 5]{Seuret_2015} it is shown that}
\begin{split}
    \Delta V_{c}(k) \leq \xi^{\T}(k)\Phi_3(d_m,d_M)\xi(k),
\end{split}\label{eqWirtinger}
\end{align}
\end{subequations}where $\Phi_3(d_m,d_M)$, defined in Lemma \ref{lem:stab}, contains the matrix ${\Psi}_z$, which has to be symmetric positive definite and is composed by the LKF matrix $\Z_2$ and by the slack decision variable $\Xm$ in $\reals^{2n\times2n}$. By combining \eqref{eqP}, \eqref{eqQ} and \eqref{eqWirtinger}, the bound $\Delta V(k) \leq \xi^{\T}(k){\Phi}(d(k))\xi(k)$ is obtained, where ${\Phi}(d)$ is defined in Lemma \ref{lem:stab}.
Then, from Lemma~\ref{lemma:Finsler}, satisfaction of 
\begin{align*}
            &\xi^{\T}(k){\Phi}(d(k))\xi(k)<0, ~\forall\, \xi \neq 0, \\
            &\text{such that} \;\underbrace{\begin{bmatrix}
-\I & \hspace{-0.1cm}\A & \hspace{-0.1cm} \A_m & \hspace{-0.1cm} \Ad & \hspace{-0.1cm} \A_M & \hspace{-0.1cm} 0_{n \times 3n}
\end{bmatrix}}_{\Gamma} \xi=0, 
\end{align*}
(and therefore of $\Delta \V(k) <0$) along the trajectories of \eqref{systemswithupperbounds} is equivalent to the satisfaction of ${\Gamma^{\perp}}^{\T}{\Phi}(d(k)){\Gamma^{\perp}} \prec 0$, where ${\Gamma^{\perp}}$ is such that the relation $ \Gamma {\Gamma^{\perp}} = 0$ holds. Since ${\Gamma^{\perp}}^{\T}{\Phi}(d(k)){\Gamma^{\perp}} \prec 0$ is affine with respect to $d(k)$, it suffices to evaluate it for $d_m$ and $d_M$, thus completing the proof. \end{proof}


Concerning the choice of the LKF structure in~\eqref{LKFFunc}-\eqref{eq:Lyapunovstrcture} and its manipulation in Lemma \ref{lem:stab}, we have chosen to use the strategy from \cite{Seuret_2015} due to the application of the Wirtinger inequality which yields less conservatism than those applying the classical inequalities as Jensen's one \cite{Lin_2008}. The last decades have witnessed intense research involving the proposal and application of less conservative summation inequalities for various structures of LKFs. The classical Jensen's inequality has been replaced by Wirtinger and Bessel inequalities, which yield enhanced results in the stability analysis of time-delayed systems. Although Bessel's one has been shown to yield the best results so far \cite{SEURET20151}, we decided to present the methodology in this work using the Wirtinger one due to better readability. Notice that the use of slack variables introduced by proposition \ref{lemma:Finsler:item2} of Lemma \ref{lemma:Finsler} is not mandatory for analysis purposes. However, the use of slack variable $\Jf$ in proposition \ref{lemma:Finsler:item2} of Lemma \ref{lemma:Finsler} will be particularly useful when solving the controller design problem in Section~\ref{sec:observerdesign}.

\section{Switched representation and path-complete criteria}\label{sec:witchform}
\subsection{Description of the considered delay-switched systems}
\textcolor{black}{We introduce an alternative representation of~\eqref{clsystem} in order to adapt the techniques presented in Section~\ref{sec:main} in this setting. In particular, we rewrite system~\eqref{clsystem} in the form of a switched system composed by two delay subsystems, each one of the form~\eqref{systemswithupperbounds}.} This is possible by noting that, when the time-varying delay $d(k)$ is between the minimum delay $d_m$ and the constant delay $d_n$, $d_n$ can be viewed as the maximum system delay for a first subsystem; when $d(k)$ is between $d_n$ and the maximum delay $d_M$, $d_n$ can be viewed as the lower bound on the delay for a second subsystem. Therefore, the following two-modes switched system can be defined:
\begin{equation} \label{clsystemswitch}
\begin{cases}
x(k+1) = f_{\sigma(k)}(\,\xbar{x}(k),k),~\forall k\geq0\\
x(k) = \phi(k), ~\forall k \in \left[-d_M,0 \right] 
          \end{cases}
\end{equation}
considering, for $j\in \{1,2\}$, the subsystems $f_j:\reals^{n\times (d_M+1)} \times \naturals \to \reals^n$ defined by
\begin{equation}\label{eq:Defnsubsystemf}
\begin{aligned}
f_j(\,\xbar{x}(k),k):=&\A x(k) + \A_{m_j} x(k-d_{m_j})\\&+ \A_{M_j} x(k-d_{M_j}) + \Ad x(k-d(k)),
\end{aligned}
\end{equation}
where 
\begin{equation}\label{eqswitchdefinitions}
\begin{array}{l|l|l|l}
   \Amone=0  &  \AMone=\An & \dmone=d_m &\dMone=d_n\\
  \addlinespace[1pt] \hline \addlinespace[1pt]
  \Amtwo=\An   & \AMtwo=0& \dmtwo=d_n& \dMtwo=d_M 
\end{array}
\end{equation}
The \emph{switching rule} $\sigma:\naturals\to \{1,2\}$ is defined by
\begin{equation}\label{sigmafunc}
    \sigma(k) =\begin{cases}
    1,\;\;\;\; \text{ if } d_m \leq d(k) \leq d_n, \\
    2,\;\;\;\;  \text{ if } d_n< d(k) \leq d_M.
    \end{cases}
\end{equation}
  With this convention, the time-varying delay $d(k)$ is  such that $\dms \leq d(k) \leq \dMs$, for any $k\in \naturals$.
That means that the delay lower bound (noted $\dms$) and the delay upper bound (noted $\dMs$) depend on the value of $\sigma(k)$ at each sampling time. The main advantage of this representation is that we get rid of the constant delay $d_n$, which is now seen as an upper/lower bound (depending on the active subsystems) for the time-varying delay $d(k)$. This allows us to restore in this setting the manipulation techniques described in Section~\ref{sec:main} for the construction of Lyapunov-Krasovskii functionals even if the resulting system \eqref{clsystemswitch} is, at this stage, a switched system. However, this is not trivial since, as will be shown in the proof of Theorem \ref{thm:stab}, a new LKF modified from \eqref{LKFFunc} along with new manipulations are proposed. Achieving convex conditions for the stability analysis of the equivalent switched system with arbitrary delay $d(k)$ defined in \eqref{clsystemswitch}-\eqref{sigmafunc} is, thus, our first contribution.

\subsection{Stability Analysis}
\textcolor{black}{
The search for a \emph{common} Lyapunov-Krasovskii functional for the two delay subsystems defined by 
$f_1$ and $f_2$ can be computationally hard, leading to conservative or even structurally  infeasible conditions.
We propose a less conservative construction based on \emph{multiple} Lyapunov-Krasovskii functionals, see~\cite{Bra98} for an overview in the delay-free case. In particular we adapt, in this context, the ideas introduced in~\cite{PhiEss16,AhmJun14}, proposing a graph-based structure for the multiple Lyapunov-Krasovskii functionals, formally introduced in the following statement and graphically represented in Fig.~\ref{fig:pcg_n_0}.}
\begin{figure}[t!]
  \centering
\begin{tikzpicture}%
  [>=stealth,
   shorten >=4pt,
   node distance=3.5cm,
   on grid,
   auto,
   every state/.style={draw=blue!70, fill=blue!10, thick}
  ]
\node[state,inner sep=1pt, minimum size=20pt] (left)                  {{ $V_1$}};
\node[state,inner sep=1pt, minimum size=20pt] (right) [right=of left] {{$V_2$}};
\path[->]
   (left) edge[bend left=20]     node          [scale=0.9]            {$f_2$} (right)
        (right)   edge[bend left=20] node        [scale=0.9]        {$f_1$} (left)
   (left) edge[loop left=40]     node            [scale=0.9]          {$f_1$} (left)
   (right) edge[loop right=40]     node      [scale=0.9]                {$f_2$} (right)
   ;
\end{tikzpicture}
\caption{Path-complete representation of inequalities in~\eqref{eq:EdgesInequality} in Lemma~\ref{lem:pathcompletecriteria}.}
  \label{fig:pcg_n_0}
\end{figure}
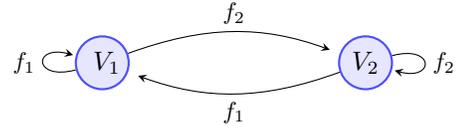
\begin{lem}\label{lem:pathcompletecriteria}
Suppose there exist positive definite $\V_j: \reals^{n\times(d_M+1)} \to \reals$, $j \in \{1,2\}$, such that the following conditions are satisfied, for any $\xbar{x}(k)\neq 0$, 
\begin{subequations}\label{eq:EdgesInequality}
\begin{align}\label{EdgeIneq1}
    \V_1(k+1)&< \V_1(k),\;\;\text{with }~ x(k+1) = f_1(\xbar{x}(k),k),
\\\label{EdgeIneq2}
     \V_2(k+1)&< \V_2(k),\;\;\text{with }~ x(k+1) = f_2(\xbar{x}(k),k),
\\\label{EdgeIneq3}
    \V_1(k+1)&< \V_2(k),\;\;\text{with }~ x(k+1) = f_1(\xbar{x}(k),k),\\
    \label{EdgeIneq4}
    \V_2(k+1)&< \V_1(k),\;\;\text{with }~ x(k+1) = f_2(\xbar{x}(k),k),
\end{align}
\end{subequations}
using the notation introduced in~\eqref{eq:LKfunctional} and $f_j$ defined by~\eqref{eq:Defnsubsystemf} and~\eqref{eqswitchdefinitions}.
Then, system \eqref{clsystemswitch} is asymptotically stable  for all time-varying delays $d:\naturals \to [d_m,d_M]$ and $\min_{j\in \{1,2\}}\{V_j(\xbar{x})\}$ is a LKF for \eqref{clsystemswitch}. 
\end{lem}
\begin{proof}[Sketch of the Proof]
The proof follows common ideas from path-complete Lyapunov theory for delay-free switched systems, see for example~\cite{AhmJun14,PhiEss16} to which we refer for the details. We give in what follows the geometric intuition underlying any path-complete stability criteria. In Fig.~\ref{fig:pcg_n_0}, we depicted a graph-based representation of inequalities in~\eqref{eq:EdgesInequality}: a generic edge $(\V_i,\V_j,f_j)$, with $(i,j)\in \{1,2\}^2$ represents the inequality $
\V_j(k+1)<V_i(k)$, with $x(k+1) = f_j(\xbar{x}(k))$.
In terms of sublevel sets in the extended state space $\reals^{n\times(d_M+1)}$, this inequality implies that \emph{any} sublevel set of $\V_i$ is mapped, by the subsystem $f_j$, inside the corresponding sublevel set of $\V_j$. We then note that the graph in Fig.~\ref{fig:pcg_n_0} (corresponding to the inequalities in~\eqref{eq:EdgesInequality}) is \emph{path-complete}, i.e., for any finite sequence of $\{1,2\}$, there exists a directed path labeled by the considered sequence. This implies that, \emph{for any} time-varying delay signal $d:\naturals\to [d_m,d_M]$ (and thus, recalling~\eqref{sigmafunc}, for any signal $\sigma:\naturals\to \{1,2\}$) the value of the (multiple) Lyapunov-Krasovskii functional (following the ``path'' corresponding to $\sigma$) is strictly decreasing along the trajectories of~\eqref{clsystemswitch}, implying asymptotic stability.
\end{proof}
In the following statement we propose LMI-based conditions, leading to \emph{quadratic} Lyapunov-Krasovskii functionals $\V_1,\V_2$ satisfying the conditions~\eqref{eq:EdgesInequality}
 of Lemma~\ref{lem:pathcompletecriteria}, thus ensuring asymptotic stability of system~\eqref{clsystem}. For better readability of the conditions, auxiliary matrices are defined in Appendix~\ref{SecondAppendix}.


\textcolor{black}{\begin{thm}\label{thm:stab}
Consider $j \in \{1,2\}$ and assume that there exist matrices ${\Pj} \in\mathbb{S}_{3n+n(2-j)}^{+}$, $\Qonej$, $\Qtwoj$, $\Qthree$, $\Zonej$, $\Ztwoj$, $\Zthree \in\mathbb{S}_{n}^{+}$, ${\Xm_{_j}}\in\reals^{2n\times2n}$ such that 
\begin{equation}\label{eqmain_teo}
{\Psi}_{z_{_j}} \succ 0,
\!\begin{aligned}
~~{\Gamma_{j}^{\perp}}^{\T}{\Phi}(\dmj,j){\Gamma_{j}^{\perp}} \prec 0, ~~{\Gamma_{j}^{\perp}}^{\T}{\Phi}(\dMj,j){\Gamma_{j}^{\perp}} \prec 0,
\end{aligned}
\end{equation}
\noindent hold with 
\begin{equation*}
    {\Gamma_{j}^{\perp}} =
  \begin{bmatrix} 
  \A & \Amj & \Ad & \hspace{-1cm} \AMj &  \hspace{-1cm} 0_{n \times 3n} & \hspace{-0.1cm} 0_{n(2-j) \times 2n} \\
 & & & \I_{7n+2n(2-j)} & &
   \end{bmatrix},
\end{equation*}
\begin{align*}
    &{\Phi}(d,j) = \Phi_1(d,j)+{\mathcal{Q}_{_j}}
   + \Phi_{3}(j), \\
   &\Phi_1(d,j)=\W_{2_j}^{\T}{\Pj}\W_{2_j}-\W_{1_j}^{\T}{\Pj}\W_{1_j} \\&~~~~~~~~~~~~+ \text{He}\left\{\W_j^{\T}(d){\Pp_j}\left(\W_{2_j}-\W_{1_j}\right)\right\}, \\
   &\Phi_3(j)=\W_{3_j}^{\T}(\dmj^2 {{\Zonej}}+\hdj^2 {{\Ztwoj}}+\mathcal{Z}_{3_j}) \W_{3_j}\\
   &~~~~~~-\W_{s_j}^{\T} \mathscr{Z}_{1_{j}}(\dmj)\W_{s_j}-\W_{\Psi_j}^{\T} {\Psi}_{z_{_j}}\W_{\Psi_j}-\W_z^{\T} \mathscr{Z}_{3_{j}} \W_z,
\end{align*}
\noindent where $\hdj=\dMj-\dmj$, $\mathcal{Z}_{3_j}=(2-j)d_{\Delta_M}^2\Zthree$, $d_{\Delta_M} = d_M-d_{{M_j}}$,
\begin{equation*}
    {\Psi}_{z_{_j}} = \begin{bmatrix}
    ~\mathscr{Z}_{2_{j}} & {\Xm_{_j}}~ \\
    ~\star & \mathscr{Z}_{2_{j}}~
    \end{bmatrix}, \mathscr{Z}_{2_{j}} = \text{diag}\left({\Ztwoj},3{\Ztwoj}\right),
\end{equation*}
\begin{equation*}
    \mathscr{Z}_{1_{j}}(\dmj) = \text{diag}\left({\Zonej},3\gamma(\dmj){\Zonej}\right),
\end{equation*}
\begin{align*}
\mathscr{Z}_{3_{j}} \hspace{-0.1cm}= & (2-j) \text{diag}\left({\Zthree},3{\Zthree}\right),\\
    {\mathcal{Q}_{_j}} \hspace{-0.1cm}= &\text{diag}\left(0,{\Qonej}, {\Qtwoj}-{\Qonej},0,\R_j\right), \\
    \R_j \hspace{-0.1cm}= &\text{diag}\left( \Qthree \hspace{-0.1cm}(2-j)-{\Qtwoj}, 0_{\left( 3n+(2-j)2n \right) \times \left( 3n+(2-j)2n \right)} \right),
\end{align*}
at the same time that 
\begin{align}\label{extra:ineq1}
\begin{split}
\mathcal{L}^{\perp^{\T}}(l_1) \left[\textit{diag}(\Ss_1,0_n) - \textit{diag}(0_n,\Ss_2)\right] \mathcal{L}^{\perp}(l_1) \prec 0,\\l_1 = d_{m_1}, ..., d_{M_1},
\end{split} 
\end{align}
\begin{align}\label{extra:ineq2}
\begin{split}
    \mathcal{L}^{\perp^{\T}}(l_2) \left[\textit{diag}(\Ss_2,0_n) - \textit{diag}(0_n,\Ss_1)\right] \mathcal{L}^{\perp}(l_2) \prec 0,\\l_2 = d_{m_2}, ..., d_{M_2}, 
\end{split}
\end{align}
\noindent with $S_j = \W_{5_{j}}^\T \Pp_{_j} \W_{5_{j}} + \mathcal{P}_{_j}$ is also satisfied. Then, 
system \eqref{clsystemswitch} is asymptotically stable for any time-varying delay $d:\naturals\to [d_m,d_M]$. 
\end{thm}}

\begin{proof}
In order to establish stability of \eqref{clsystemswitch} we need to find LKFs $V_1$ and $V_2$ that satisfy \eqref{eq:EdgesInequality} in Lemma~\ref{lem:pathcompletecriteria}.

\textcolor{black}{First of all, to verify conditions \eqref{EdgeIneq1} and~\eqref{EdgeIneq2} we can apply the result in  Lemma~\ref{lem:stab}, since  subsystems $x(k+1)=f_j(\xbar{x}(k))$ defined in~\eqref{eq:Defnsubsystemf} are in the form of~\eqref{systemswithupperbounds}, for each $j\in \{1,2\}$. Since the stability of the overall switched delay system \eqref{clsystemswitch} depends on the whole state history $\xbar{x}(k)$, we need to define a LKF $V_j:\reals^{n \times (d_M+1)} \rightarrow \reals ^+$ for each $j\in \{1,2\}$. For $j=2$, consider the LKF of the form
\begin{equation}\label{LKFFuncSwitched2}
     V_2(k) = V_{a_2}(k)+V_{b_2}(k)+V_{c_2}(k),
 \end{equation}
where $V_{a_2},V_{b_2}, V_{c_2}$ are defined as in~\eqref{eq:Lyapunovstrcture} for $j=2$, \emph{mutatis mutandis}. By defining the extended vector $\xi_2(k)$ where $\xi_j(k)$ has the same structure of $\xi(k)$ but written in function of the minimum and maximum delays of the second subsystem, i.e., $d_{m_2}$ and $d_{M_2}$, and applying the manipulations in Lemma~\ref{lem:stab}, we arrive at the conditions~\eqref{eqmain_teo} in Theorem \ref{thm:stab} for the case of $j=2$, therefore having a sufficient condition for \eqref{EdgeIneq2}.} 

\textcolor{black}{Next, we find a LMI sufficient condition to verify conditions~\eqref{EdgeIneq1}. Note that in this case, if we consider $V_1$ in the same format of $V$ i.e. depending on the minimum and maximum delays of the first subsystem $d_{m_1}$ and $d_{M_1}$, $V_1$ would map $\reals^{n \times (d_n+1)} \rightarrow \reals ^+$ instead of $\reals^{n \times (d_M+1)} \rightarrow \reals ^+$, preventing the possibility of having ``composite'' conditions involving $\V_1$ and $\V_2$, as required by Lemma~\ref{lem:pathcompletecriteria}. Therefore, we propose a modification to the structure of $V_1$; We thus define: 
\begin{equation}\label{LKFFuncSwitched1}
     V_1(k) := V_{a_1}(k)+V_{b_1}(k)+V_{c_1}(k),
 \end{equation}
 where $V_{a_1}$, $V_{b_1}$, and $V_{c_1}$ have a form similar to~\eqref{eq:Lyapunovstrcture} but with the changes described below.
 \begin{itemize}[leftmargin=*]
     \item $V_{a_1}$ is defined as $V_{a}(k) = w_1^{\T}(k)\Pp_1 w_1(k)$, with $\Pp_1 \in\mathbb{S}_{4n}^{+}$ and the vector
     \begin{equation*}
        \hspace{-0.3cm} w_1(k) \hspace{-0.05cm} := \hspace{-0.05cm} \begin{bmatrix} x^{\T}(k) & \hspace{-0.2cm}\sum\limits_{l=k-d_{m_1}}^{k-1} \hspace{-0cm}\hspace{-0.4cm}x^{\T}(l) & \hspace{-0.2cm}\sum\limits_{l=k-d_{M_1}}^{k-d_{m_1}-1} \hspace{-0.4cm}x^{\T}(l) & \hspace{-0.2cm}\sum\limits_{l=k-d_{M}}^{k-d_{M_1}-1} \hspace{-0.4cm}x^{\T}(l) \end{bmatrix}^{\T}\hspace{-0.05cm}.
     \end{equation*}
     \item $V_{b_1}$ has the form of $V_b$, but includes a third summation term $\sum_{l=k-d_M}^{k-d_{M_1}-1} x^{\T}(l)\Q_{3_1} x(l)$. 
     \item $V_{c_1}$ has the form of $V_c$, but also includes an additional summation term, given by $(d_M-d_{M_1}) \sum_{l=-d_M+1}^{-d_{M_1}} \sum_{i=k+l}^{k} \eta^{\T}(i)\Z_{3_1}\eta(i)$. 
 \end{itemize}
The idea behind the modifications in $V_1$ is to include summation terms between $d_{M_1}=d_n$ (which is the maximum delay for subsystem $1$) and the maximum delay of the overall switched system \eqref{clsystemswitch}  (which is given by $d_M$) where the relation $d_M  \geq d_n$ holds. Such modifications lead to additional matrices $\Q_{3_1}$ and $\Z_{3_1}$, required only for the mode $1$.}
We define extended vector $\xi_{1}(k)$ in the same fashion of $\xi(k)$, but written replacing the delay limits $d_m$ and $d_M$ by $d_{m_1}$ and $d_{M_1}$, respectively, and also by adding the terms $x(k-d_M)$ and $v_4 = \frac{1}{d_M-d_{M_1}+1}\sum_{l=k-d_M}^{k-d_{M_1}}x(l)$. Once again, by applying the manipulations studied in \cite{Seuret_2015} and reviewed in the proof of Lemma \ref{lem:stab}, a sufficient condition to ensure the satisfaction of \eqref{EdgeIneq1} is obtained as described by~\eqref{eqmain_teo}  with $j=1$ in Theorem \ref{thm:stab}. 
 To summarize, inequalities~\eqref{eqmain_teo} for $j\in \{1,2\}$, are sufficient conditions for the existence of positive definite LKFs $V_j$, $j \in \{1,2\}$, that satisfy~\eqref{EdgeIneq1} and~\eqref{EdgeIneq2} , i.e., the conditions which are graphically illustrated by the ``self-loops" in Fig.~\ref{fig:pcg_n_0}. It remains to show that inequalities~\eqref{EdgeIneq3} and~\eqref{EdgeIneq4} are also fulfilled by the conditions in Theorem \ref{thm:stab}.
 
 \textcolor{black}{For this end, first note that the functionals $V_j$, $j \in \{1,2\}$, given in \eqref{LKFFuncSwitched1} and \eqref{LKFFuncSwitched2} can be rewritten in the form $V_j(k)=\xbar{x}^{\T}(k) \left(\W_{5_{j}}^\T \Pp_{_j} \W_{5_{j}} + \mathcal{P}_{_j} \right)\xbar{x}^{\T}(k)$, with matrices $\W_{5_{j}}$ and $\mathcal{P}_{_j}$ given in Appendix~\ref{SecondAppendix}. Then, inequalities
 \begin{equation}\label{eq:CondForEdgeIneq3}
   \begin{split}
        \kappa^\T \left[\textit{diag}(\Ss_1,0_n) - \textit{diag}(0_n,\Ss_2)\right] \kappa < 0,
        \\\forall~\mathcal{L}(l_1) \kappa = 0, \kappa \neq 0,~l_1 = d_{m_1}, ..., d_{M_1}
   \end{split}
 \end{equation}
 \begin{equation}\label{eq:CondForEdgeIneq4}
     \begin{split}
        \kappa^\T \left[\textit{diag}(\Ss_2,0_n) - \textit{diag}(0_n,\Ss_1)\right] \kappa < 0,
        \\\forall~\mathcal{L}(l_2) \kappa = 0, \kappa \neq 0,~l_2 = d_{m_2}, ..., d_{M_2}
   \end{split}
 \end{equation}
  with $\kappa(k) = \begin{bmatrix}
 x^\T (k+1) & \xbar{x}^{\T}(k)
 \end{bmatrix}^{\T}$, imply conditions~\eqref{EdgeIneq3} and~\eqref{EdgeIneq4}, respectively.  We note that the dynamics $x(k+1)=f_j(\xbar{x}(k))$, for any $d_{m_j} \leq d(k) \leq d_{M_j}$, are expressed by the restriction $\mathcal{L}(l_j) \kappa = 0, \kappa \neq 0,~l_j = d_{m_j}, ..., d_{M_j}$. Then, by applying form~\ref{lemma:Finsler:item3} of Lemma~\ref{lemma:Finsler} conditions~\eqref{extra:ineq1}-\eqref{extra:ineq2} of Theorem \ref{thm:stab} are obtained, completing the proof.}
 \end{proof}

\subsection{Comparison with alternative existing approaches}

An alternative modeling strategy for system~\eqref{clsystem} is to consider a corresponding (delay-free) \emph{linear switched system} in dimension $n\times (d_M+1)$ with $d_M-d_m$ modes, considering the augmented state $\xbar{x}(k)\in \reals^{n\times (d_M+1)}$ defined in~\eqref{eq:LKfunctional}. This idea is explored, in similar settings (but without the term depending on the constant delay $d_n$), among others, in \cite{HETEL2008,8793178,Vittorio2020}. Classical techniques from switched systems literature can then be used to establish stability of~\eqref{clsystem}.
 However, some disadvantages would appear with respect to the modeling technique described in this section:
 \begin{itemize}[leftmargin=*]
 \item For large delays, the system would be of high order, possibly leading to numerical issues related to the curse of dimensionality and the sparsity of the arising matrices.
 \item The manipulation for obtaining LMIs for the design of \emph{controller and observer gains} (presented in Section~\ref{sec:observerdesign}) would become involved or even infeasible, due to the presence of the additional decision variables arising from the control and observer gains, as detailed in what follows.
 \end{itemize}
This latter drawback is particularly important since the main motivation in studying systems of the form~\eqref{clsystem} arises from the control design problem for a class of systems with output delay, which is the main goal of the next section.

\section{Application to observer-controller design} \label{sec:observerdesign}

In this section, with the aid of form (ii) of Lemma~\ref{lemma:Finsler}, we extend the conditions developed in the previous section to \emph{co-design}, by means of LMIs, an observer and a controller  for the stabilization of systems with output time-varying delays. 

\subsection{Plant and controller description}
Consider the plant described by the following equations
\begin{align}
    \begin{cases}\label{eq:Plant}
      x_{p}(k+1)={\A}_p x_{p}(k)+{\B}_p  u(k) \\
      y(k)= x_{p}(k-d(k))\\
    \end{cases} 
    \end{align}
\noindent where $x_{p}(k) \in \reals ^{n_p}$ is the plant state vector, $y(k) \in \reals^{n_p}$ is the delayed measured output and $u(k) \in \reals^{m}$ is the control input. Matrices $\A_p$, $\B_p$ are constant, known, and of appropriate dimensions, and the pair ($\A_p$, $\B_p$) is controllable. The plant output delay is bounded and time-varying, satisfying $d_m \leq d(k) \leq d_M$. Furthermore, integers $d_m$ and $d_M$ are known, whereas the value of $d(k)$ at each sampling time is unknown. 
To control system \eqref{eq:Plant} we consider the following Luenberger-type observer plus control pair
\begin{align}
    \begin{cases}\label{eq:control}
      \xp_{p}(k+1)={\A}_p \xp_{p}(k)+{\B}_p  u(k) + \Lm e_y(k)\\
      u(k)= \K \xp_{p}(k)+\F e_y(k)\\
    \end{cases} 
    \end{align}
\noindent where $e_y(k) = y(k)-\xp_{p}(k-d_{n})$, $d_n$ is the constant nominal delay for the observer, the term $\Lm$ is the classical observer corrector term, and $\K$ is related to state feedback control of the observed state. The matrix $\F$, which filters the observer error in the control law, is an important extra degree of freedom to stabilize system~\eqref{eq:Plant} and can be viewed as a static version of the robustness filter in the Filtered Smith Predictor (FSP) strategy \cite{Normeyrico_2007}. Since the time-varying  delay $d(k)$ is unknown, the constant delay $d_n\in [d_m,d_M]$ represents, in this setting, a ``guess'' for $d(k)$ the designer provides to the observer-based controller. In the uncertain delay case, such a constant delay is traditionally chosen as the mean between maximum and minimum delays in predictive/observer-based delay compensation strategies. However, such a choice can be further explored and provides a new significant degree of freedom, as will be explored in the numerical examples.



\subsection{Closed-loop system and problem formulation}
Consider the error signal defined by
\begin{equation}\label{eq:error}
    e(k)=x_{p}(k)-\xp_{p}(k).
\end{equation}
\noindent By taking into account \eqref{eq:Plant}, \eqref{eq:control}, and by defining the extended vector $x(k) = \begin{bmatrix} x_{p}(k)^{\T} & e(k)^{\T} \end{bmatrix}^{\T} \in \reals^{n}$, $n=2n_p$, the closed-loop system can be written as \eqref{clsystem} with 
\begin{align}
          \begin{array}{l}
 \A=\begin{bmatrix} 
\A_p+\B_p\K & -\B_p\K \\
0 & \A_p
\end{bmatrix}, \Ad=\begin{bmatrix} 
\B_p\F & 0 \\
-\Lm & 0 
\end{bmatrix},\\
\An=\begin{bmatrix} 
-\B_p\F & \B_p\F \\
\Lm & -\Lm 
\end{bmatrix}.
\end{array}
\label{eq:matrices}
\end{align}
\noindent  The problem we study is summarized in what follows.
\begin{problem}
Given the plant matrices $\A_p$, $\B_p$, and the time-varying delay limits $d_m$, $d_M$, develop convex conditions in the form of LMIs for the design of matrices $\K$, $\Lm$ and $\F$, such that the asymptotic stability of the closed-loop system \eqref{clsystem}-\eqref{eq:matrices} is ensured for any time-varying delay $d_m \leq d(k) \leq d_M$. 
\label{prob1}
\end{problem}

\subsection{Observer-based controller design}
The following corollary provides a solution to Problem \ref{prob1} \textcolor{black}{by leveraging the stability conditions developed in Section~\ref{sec:witchform}}.
\textcolor{black}{\begin{cor} \label{cor:stabilization}
Given scalar $\varepsilon\in[0,-1)$,  assume that, for any $j \in \{1,2\}$, there exist matrices ${\Pjb} $ in $\mathbb{S}_{3n+(2-j)n}^{+}$, $\Qonejb$, $\Qtwojb$, $\Qthreeb$, $\Zonejb$, $\Ztwojb$, $\Zthreeb$ in $\mathbb{S}_{n}^{+}$, ${\xbar{\Xm}_{_j}}$ in $\reals^{2n\times2n}$, $\J=$\textit{diag}$(\U,\U)$ in $\reals^{n\times n}$, $\U$ in $\reals^{n_p\times n_p}$, $\xbar{\K}$, $\xbar{\F}$ in $\reals^{m\times n_p}$, and $\xbar{\Lm}$ in $\reals^{n_p\times n_p}$ such that 
\begin{equation} \label{main_matineq}
    \xbar{\Psi}_{z_{j}} \succ 0, \xbar{\Phi}(\dmj,j)+
    \xbar{\Upsilon}_j \prec 0, \xbar{\Phi}(\dMj,j)+\xbar{\Upsilon}_j \prec 0
\end{equation}
\noindent where $\xbar{\Phi}(d,j)$ has the same format of ${\Phi}(d,j)$ in Theorem~\ref{thm:stab} but with the ``bar" matrices and where 
\begin{equation*}
    \xbar{\Upsilon}_j = \text{diag} \left(\xbar{\Upsilon}_{a_j}, 0_{\left( 3n+(2-j)2n \right) \times \left( 3n+(2-j)2n \right)} \right), 
\end{equation*}
\begin{equation*}
  \xbar{\Upsilon}_{a_j}  =
\left[
\setlength\arraycolsep{2pt}
\begin{array}{cccccc}
 -\J-\J^{\T} & \xbar{\A}-\varepsilon\J & \Amjbar & \xbar{\Ad} & \AMjbar  \\
    \star & \varepsilon\xbar{\A}+\varepsilon \xbar{\A}^{\T}\hspace{-0.1cm}& \varepsilon\Amjbar  & \varepsilon\xbar{\Ad}  & \varepsilon \AMjbar   \\ \star & \star & 0 & 0 & 0  \\
    \star & \star & \star & 0 & 0  \\
    \star & \star & \star & \star & 0 
\end{array}
\right],
\end{equation*}
\noindent with
\begin{align*}
&\xbar{\A}=\begin{bmatrix} 
\A_p\U^{\T}+\B_p\xbar{\K} & -\B_p\xbar{\K} \\
0 & \A_p\U^{\T}
\end{bmatrix}, ~\xbar{\Ad}=\begin{bmatrix} 
\B_p\xbar{\F} & 0 \\
-\xbar{\Lm} & 0 
\end{bmatrix},
\end{align*}
\begin{equation*}
    \begin{cases}
    \Amjbarone = 0, ~\Amjbartwo = \xbar{\An} \\
    \AMjbarone = \xbar{\An}, ~\AMjbartwo = 0
    \end{cases}, ~~~ \xbar{\An}=\begin{bmatrix}
-\B_p\xbar{\F} & \B_p\xbar{\F} \\
\xbar{\Lm} & -\xbar{\Lm}
\end{bmatrix},
\end{equation*}
\noindent hold at the same that the following inequalities are also satisfied
\begin{align}\label{eqcor:extra:ineq1}
\begin{split}
\left[\textit{diag}(\xbar{\Ss}_1,0_n) - \textit{diag}(0_n,\xbar{\Ss}_2)\right] +\text{He}\{\mathcal{I}^\T \xbar{\mathcal{L}}(l_1)\} \prec 0,\\l_1 = d_{m_1}, ..., d_{M_1},
\end{split} 
\end{align}
\begin{align}\label{eqcor:extra:ineq2}
\begin{split}
\left[\textit{diag}(\xbar{\Ss}_2,0_n) - \textit{diag}(0_n,\xbar{\Ss}_1)\right] +\text{He}\{\mathcal{I}^\T \xbar{\mathcal{L}}(l_2)\} \prec 0,\\l_2 = d_{m_2}, ..., d_{M_2},
\end{split} 
\end{align}
where $\mathcal{I}=\begin{bmatrix}
\I & \varepsilon \I & 0 & \dotsb & 0
\end{bmatrix}$ and
    \begin{align*}
    \xbar{\mathcal{L}}(l) =& \begin{bmatrix}
    -\J & \xbar{\A} & 0 & \dotsb & 0 & \xbar{\An} & 0 & \dotsb &  0
    \end{bmatrix}\\ &+ 
    \xbar{\Ad} \begin{bmatrix} 0 & 0 & \delta(1) \I & \dotsb & \delta(d_M) \I
    \end{bmatrix},~l \in [1,d_M],\\ &\text{with } \delta(i) = 1 \text{ if } i=l, \text{ and } \delta(i) = 0 \text{ if } i \neq l.
\end{align*}
Then, matrices $\K=\xbar{\K}\U^{-T}$, $\F=\xbar{\F}\U^{-T}$, $\Lm=\xbar{\Lm}\U^{-T}$, are such that the closed loop yielded by the connection between the plant \eqref{eq:Plant} and the observer-controller \eqref{eq:control}, i.e., system ~\eqref{clsystem} with matrices given as in~\eqref{eq:matrices}, is asymptotically stable for any time-varying delay $d:\naturals\to [d_m,d_M]$. 
\end{cor}
}
\begin{proof}

Consider the LKFs $V_j$ as in~\eqref{LKFFuncSwitched1}-\eqref{LKFFuncSwitched2}, the extended vector $\xi_{_j}(k)$ and the matrix function ${\Phi}(d,j)$. The bound $\Delta V_j(k) \leq \xi_{_j}^{\T}(k){\Phi}(d(k),j)\xi_{_j}(k)$, for each $j \in \{1,2\}$ holds by previously commented arguments. 

Next, we apply Lemma~\ref{lemma:Finsler}. By noting that $\Gamma_j \xi_j(k) = 0$ with $\Gamma_j = \begin{bmatrix} -\I & \A & \Amj & \Ad & \AMj & \hspace{-0.1cm} 0 & \hspace{-0.1cm} \dotsb & \hspace{-0.1cm} 0 \end{bmatrix}$, the relation $\xi_j^{\T}(k){\Phi}(d(k),j)\xi_j(k) \prec 0$ is verified for all $\xi_j(k) \neq 0$ if and only if there exist matrices $\Jf_j$ such that 
\begin{equation}\label{eq:differentFinslermatrices}
    {\Phi}(d(k),j)+\Jf_j \Gamma_j+\Gamma_j^{\T} \Jf_j \prec 0, \text{ for } j \in \{1,2\}
\end{equation}
Since ${\Phi}(d(k),j)$ is affine with respect to $d(k)$, the last inequality is negative definite if and only if it is negative definite for both $d(k)=\dmj$ and $d(k)=\dMj$. 
A sufficient condition ensuring the fulfillment of \eqref{eq:differentFinslermatrices} can then be obtained by defining $\Jf_1 = \Jf_2 = \Jf=\begin{bmatrix}\J^{-\T} & \varepsilon\J^{-\T}  & 0 & \cdots & 0 \end{bmatrix}^{\T},$
where $\J^{-1}=\textit{diag}(\U^{-1},\U^{-1})$ and $\varepsilon$ is an auxiliary scalar that allows some degree of freedom for the conditions. Then, by left and right multiplication of ${\Phi}(\dmj,j)+\Jf\Gamma_j+\Gamma_j^{\T}\Jf^{\T}$ (also of ${\Phi}(\dMj,j)+\Jf\Gamma_j+\Gamma_j^{\T}\Jf^{\T}$) by $\textit{diag}(\J,\dotsb,\J)$ and its transpose, respectively, by left and right multiplication of $\Psi_{z_{_j}}$ by \textit{diag}$(\J,\J)$ and its transpose, and changes of variable
\begin{align*}
        &\xbar{\Pp_j}=\textit{diag}(\J,\dotsb,\J)~\Pp_j~ \textit{diag}(\J,\dotsb,\J)^{\T}, \\
            &\xbar{\Xm}_j=\textit{diag}(\J,\J)~\Xm_j~ \textit{diag} (\J,\J)^{\T},\\
        &\xbar{\K}= \K \U^{\T}, ~\xbar{\F}= \F \U^{\T}, ~\xbar{\Lm}= \Lm \U^{\T},\\
        &\{\Qonejb, \Qtwojb, \Qthreeb\} = \J \{\Qonej,\Qtwoj,\Qthree\}\J^{\T},\\
         &\{\Zonejb, \Ztwojb, \Zthreeb\} = \J \{\Zonej,\Ztwoj,\Zthree\}\J^{\T},
        \end{align*}
\noindent conditions \eqref{main_matineq} are obtained, which need to be verified for $j \in \{1,2\}$. Finally, conditions~\eqref{eqcor:extra:ineq1} and~\eqref{eqcor:extra:ineq2} are derived by applying form~\ref{lemma:Finsler:item3} of Lemma~\ref{lemma:Finsler} to \eqref{eq:CondForEdgeIneq3} and \eqref{eq:CondForEdgeIneq4}, using the multiplier $\begin{bmatrix}\J^{-\T} & \varepsilon \J^{-\T} & 0 & \cdots & 0 \end{bmatrix}^{\T}$ and applying changes of variables, completing the proof.
\end{proof}


\subsection{Numerical example}\label{sec:simu}

Consider the NCS studied in \cite[Example 2]{HU_2007}.
By considering a sampling time of 0.5 seconds, an induced network time delay, we obtain the discrete-time model \eqref{eq:Plant} with $\A_p = \left[\begin{smallmatrix} 0.6693 & -0.0042 \\ 0.4231 & 1.0501 \end{smallmatrix}\right]$, $\B_p =\left[\begin{smallmatrix}
0.1647 \\ 0.0960
\end{smallmatrix}\right]$. In \cite{HU_2007}, the control law is given by $u(k)=-\left[\begin{smallmatrix} 1.2625 & 1.2679\end{smallmatrix}\right] x_{p}(k-d(k))$, which guarantees closed-loop stability for a maximum induced delay of 1 second (or two samples $d_M=2$), according to \cite[Theorem 4]{HU_2007}. \textcolor{black}{More recently, the compensator strategy from \cite{Lima_2021} was able to stabilize this system for a time-varying delay in the range $1 \leq d(k) \leq 7$.} By setting $\varepsilon=-0.995$ in Theorem~\ref{cor:stabilization}, $d_n=1$, we obtain matrices $\K=\left[\begin{smallmatrix} -0.1925  & -0.1702 \end{smallmatrix}\right]$, $\F=\left[\begin{smallmatrix} -0.1755 & -0.1601\end{smallmatrix}\right]$, $\Lm=\left[\begin{smallmatrix}-0.0032 & -0.0007 \\
    0.0578  & 0.0525
\end{smallmatrix}\right]$ that guarantee stability for any time-varying delay $d(k)$ such that $1 \leq d(k) \leq 17$. That means that stability is guaranteed even for a time-varying delay maximum of 8.5 seconds. That is a substantial increase compared to the results obtained by \cite{HU_2007} and \cite{Lima_2021}, and further shows the advantage of introducing the delay $d_n$ and the observer strategy instead of simply feedbacking the delayed output $x(k-d(k))$. To illustrate the effects of the LMI tuning parameter $\varepsilon$ and of the constant delay $d_n$, Fig. \ref{ex1:relationepsilonanddn} shows the maximum delay $d_M$ for which the conditions of Corollary~\ref{cor:stabilization} were feasible for different values of $\varepsilon$ and $d_n$ (the minimum delay was kept fixed at $d_m=1$). It is observed that as $\varepsilon$ approaches $-1$, the maximum delay $d_M$ tends to increase. For $d_n$, it seems that the best results are achieved near to the bounds on the delay.

\begin{figure}[httb!]
\center{\includegraphics[width=\linewidth,angle=0]{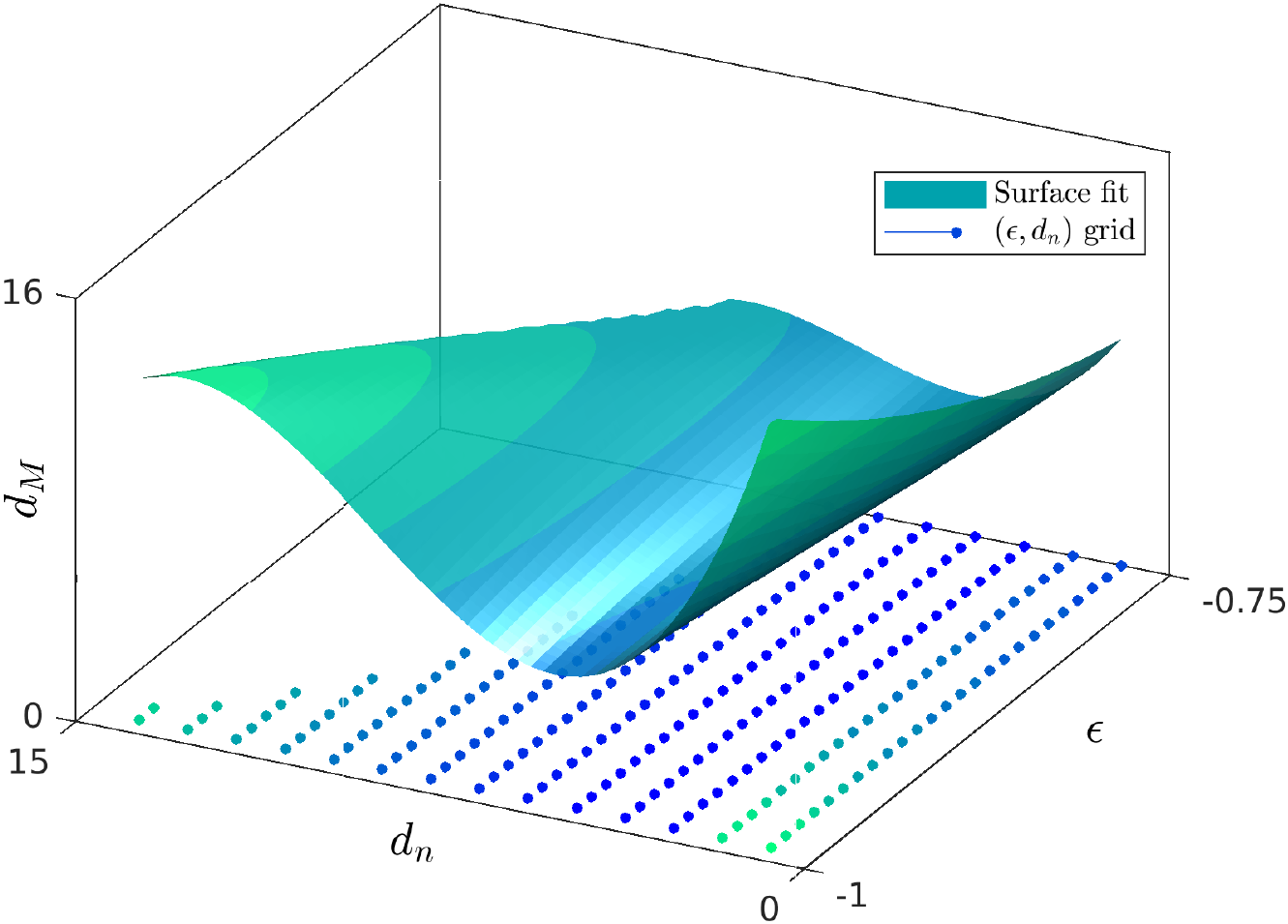}
\caption{Example 1: Relation between $\varepsilon$, $d_n$ and maximum delay $d_M$.}
\label{ex1:relationepsilonanddn}}
\end{figure}

\subsection{Discussion on observer and predictor strategies}

In the last years, the application of model-based strategies has been extensively investigated for the control of systems with delays \cite{Sanz_2019a,LI2014,Hao2017,Normeyrico_2012,ZHOU20122387,GARCIA2010367}. Recently, in \cite{Lima_2021}, analysis of a structure based on the Filtered Smith Predictor (FSP) for discrete-time systems with long output time-varying delays was studied in the presence of saturating inputs with a new methodology to characterize the regions of asymptotic stability for this type of system. In \cite{LHACHEMI2019}, the robustness of a constant-delay predictive control law for continuous-time linear systems in the case of an uncertain time-varying input delay has been accessed, while extension for a class of diagonal infinite-dimensional boundary control systems was also presented. Both \cite{Lima_2021} and \cite{LHACHEMI2019} evaluate the robustness of nominal delay predictors, that is, the predictor delay is constant and is taken between the lower and upper bounds of the plant time-varying delay in an attempt to minimize the predictor error as possible. Thus, \cite{Lima_2021} and \cite{LHACHEMI2019} have in common the fact that stabilization is not achieved by means of convex design conditions, i.e. the controller parameters are fixed and then stability analysis conditions in the form of linear matrix inequalities (LMIs) are applied.

On the other hand, the works in \cite{SanzFridman2020,SanzEmilia2018} are closer to achieving stabilization of model-based control structures for systems with uncertain time delay by means of LMI design conditions, where the use of predictive extended state observers are investigated. However, as noted in these papers, due to the difficult in linearizing the analysis conditions to the design case, a sequential approach is taken where a feedback gain $\K$ is obtained via a first condition and then other LMIs are used to obtain the parameters of the observer. Therefore, a \textit{full} design of the predictor-based controllers via LMIs is not obtained. 

In \cite{GONZALEZ2012}, sufficient conditions for the stabilization of discrete-time systems with input time-varying delays are expressed in form of LMIs. Nonetheless, a cone complementarity linearization (CCL) algorithm where the first step consists of finding a gain $\K$ that stabilizes the delay-free closed-loop is implemented and used within the numerical example section (see Section 3.1 of \cite{GONZALEZ2012}). In \cite{GONZALEZ2013}, a strategy based on the Artstein’s reduction method \cite{Artstein_1982,Manitius_1979} is used to rewrite the closed-loop system as a delay-free one interconnected to uncertainties, followed by the development of sufficient conditions to stabilize uncertain discrete-time systems with time-varying delays. Concerning observer-predictor structures, the work in \cite{Gonzalez2019} has achieved co-design through LMIs of an observer gain $\Lm$ and a feedback gain $\K$ to stabilize continuous-time linear systems with both input and output time-varying delays. Nonetheless, in order to solve a convex problem, the main design theorem in the mentioned paper requires user input of four auxiliary scalars and of a full matrix of dimension $p \times (n-p)$, where $n$ is the number of states and $p$ the number of outputs of the plant. 

Hopefully, the literature review has given the reader a sense of the difficulties in achieving stabilization of time-delayed systems with unknown and/or time-varying delays through model-based strategies using LMI-based design. In this paper, we proposed an LMI-based solution for the design of a nominal-delay observer-based control structure for discrete-time systems with time-varying delays by means of a switched-based strategy, which can be solved with user input of only one auxiliary scalar (denoted $\epsilon$ in the paper). Differently from most strategies, stabilization is achieved with the observer and control parameters being obtained in a single step by solving simple LMIs without the need to prefix any of the control gains neither use interactive approaches. This is one of the main contributions of this work. The main results are developed for the case of output delays. However, the observer strategy and the design methodology can also be extended to the cases of time-varying delays in the input and in the plant states.

\section{Conclusion}\label{sec:conclu}

In this work, a new technique for the stability analysis of systems with both a constant and a time-varying delay based on a specific switched representation has been studied in conjunction with a new modified LKF (see the proof of Theorem \ref{thm:stab}) that led to obtaining convex conditions in the form of LMIs. A stabilizing technique for discrete-time delay systems with sensor time-varying delays has been proposed as an extension. The obtained simulation results are promising, showing a substantial increase in the time-varying delay bounds for stabilization of an example from the literature. Further developments of the strategy can be envisaged to include, for example, the problem of control saturation and the study of stability conditions based on more elaborate (and possibly less conservative) path-complete criteria.




\bibliographystyle{myIEEE}
\bibliography{refs}
\appendices 
\section{Auxiliary matrices for Lemma \ref{lem:stab}}\label{FirstAppendix}
\begin{equation*}
    \W_{\Psi} = \begin{bmatrix}
   0_{2n\times2n} &\hspace{-0.2cm} & \hspace{-0.2cm}\M \\ 0_{2n\times n} & \hspace{-0.2cm}\M & \hspace{-0.2cm}0_{2n\times n}
   \end{bmatrix}, \M = \begin{bmatrix}  0 & \hspace{-0.1cm} \I \hspace{-0.1cm} & \hspace{-0.1cm} -\I & \hspace{-0.1cm} 0 \hspace{-0.1cm} & \hspace{-0.1cm} 0 & \hspace{-0.1cm} 0 \\
   0 & \hspace{-0.1cm} \I & \hspace{-0.1cm} \I & \hspace{-0.1cm} 0 \hspace{-0.1cm} & \hspace{-0.1cm} 0 & \hspace{-0.1cm} -2\I \end{bmatrix},
\end{equation*}
\begin{equation*}
    \W_{s} = \begin{bmatrix} \M & 0_{2n\times2n} \end{bmatrix}, \W_{3} = \begin{bmatrix} \I &\hspace{-0.1cm} -\I & \hspace{-0.1cm} 0 &  \hspace{-0.1cm}0 & \hspace{-0.1cm} 0 & \hspace{-0.1cm} 0 & \hspace{-0.1cm} 0 & \hspace{-0.1cm} 0 \end{bmatrix},
\end{equation*}
\begin{equation*}
    \W_{1}(d_m,d_M) \hspace{-0.1cm}= \hspace{-0.1cm}\begin{bmatrix} 0 & \hspace{-0.2cm}\I & \hspace{-0.2cm}0 &\hspace{-0.2cm} 0 & 0  &  &\\
     0 & \hspace{-0.2cm}-\I & \hspace{-0.2cm}0 & \hspace{-0.2cm}0 & 0 &  \W_{4}(d_m,d_M) &  \\
     0 &\hspace{-0.2cm} 0 & \hspace{-0.2cm}-\I & \hspace{-0.2cm}-\I & 0 & & \end{bmatrix},
\end{equation*}
\begin{equation*}
    \W_{2}(d_m,d_M) \hspace{-0.1cm}= \hspace{-0.1cm}\begin{bmatrix} \I & \hspace{-0.2cm}0 & \hspace{-0.2cm}0 &\hspace{-0.2cm} 0 & 0 & &  \\
     0 & \hspace{-0.2cm}0 & \hspace{-0.2cm}-\I & \hspace{-0.2cm}0 & 0 & \W_{4}(d_m,d_M)  \\
     0 &\hspace{-0.2cm} 0 & \hspace{-0.2cm}0 & \hspace{-0.2cm}-\I & -\I & &  
     \end{bmatrix},
\end{equation*}
\begin{equation*}
    \W_{4}(d_m,d_M) \hspace{-0.1cm}= \hspace{-0.1cm}\begin{bmatrix}0 & \hspace{-0.7cm}0 &\hspace{-0.2cm} 0 \\
     (d_m +1)\I & \hspace{-0.7cm}0 & \hspace{-0.2cm}0 \\
      0 & \hspace{-0.7cm}(1-d_m)\I &\hspace{-0.2cm} (d_M+1)\I \end{bmatrix},
\end{equation*}
\begin{equation*}
 \W(d) = \begin{bmatrix}
    0_{2n \times 8n} \\
 0_{n \times 6n}~~d \I_n~~-d \I_n
    \end{bmatrix}.
\end{equation*}

\section{Auxiliary matrices for Theorem~\ref{thm:stab}}\label{SecondAppendix}
\begin{equation*}
    \W_{\Psi_j} = \begin{bmatrix}
   \W_{\Psi} & 0_{4n \times 2n(j-2)} \end{bmatrix}, \M_d = (d_{\Delta_M}+1)\I_{n(j-2)},
\end{equation*}
\begin{equation*}
    \W_z = \begin{bmatrix} 0_{n \times 4n} & \I &  0_{n \times 3n} & -\I & 0_{n \times n} \\
   0_{n \times 4n} & \I &  0_{n \times 3n} & \I & -2\I
   \end{bmatrix},
\end{equation*}
\begin{equation*}
    \W_{s_j} = \begin{bmatrix} \W_{s} & 0_{2n \times 2n(j-2)} \end{bmatrix}, \W_{3_j} = \begin{bmatrix} \W_{3} & 0_{n \times 2n(j-2)} \end{bmatrix},
\end{equation*}
\begin{equation*}
    \W_{1_j} \hspace{-0.1cm}= \hspace{-0.1cm}\begin{bmatrix} 
     \W_{1}(d_{m_j},d_{M_j})~~0_{3n \times 2n(j-2)}\\
     0_{n \times 4n(j-2)}~-\I_{n(j-2)}~0_{n \times 4n(j-2)}~\M_d \end{bmatrix},
\end{equation*}
\begin{equation*}
    \W_{2_j} \hspace{-0.1cm}= \hspace{-0.1cm}\begin{bmatrix} 
     \W_{2}(d_{m_j},d_{M_j})~~0_{3n \times 2n(j-2)}\\
     0_{n \times 8n(j-2)}~~-\I_{n(j-2)}~~\M_d \end{bmatrix},
\end{equation*}
\begin{equation*}
 \W_j(d) = \begin{bmatrix}
    0_{2n \times 8n}~~0_{2n \times 2n(2-j)} \\
 0_{n \times 6n}~~d \I_n~~-d \I_n~~0_{n \times 2n(2-j)} \\
    0_{n \times 10n(2-j)}
    \end{bmatrix}, 
\end{equation*}
\begin{align*}
    \mathcal{L}(l) =& \begin{bmatrix}
    -\I & \A & 0 & \dotsb & 0 & \A_n & 0 & \dotsb &  0
    \end{bmatrix}\\ &+ 
    \A_d \begin{bmatrix} 0 & 0 & \delta(1) \I & \dotsb & \delta(d_M) \I
    \end{bmatrix},~l \in [1,d_M],\\ &\text{with } \delta(i) = 1 \text{ if } i=l, \text{ and } \delta(i) = 0 \text{ if } i \neq l\\
 \mathcal{L}^\perp(l) =& 
    \begin{bmatrix}
    \mathcal{L}_a^{\perp^\T}(l) & \I_{d_M+1}
    \end{bmatrix}^\T,\\
    \mathcal{L}_a^\perp(l) =& \begin{bmatrix}
    \A & 0 & \dotsb & 0 & \A_n & 0  & \dotsb &  0
    \end{bmatrix}\\ &+ 
    \A_d \begin{bmatrix} 0 & \delta(1) \I & \dotsb & \delta(d_M) \I
    \end{bmatrix},~l \in [1,d_M],\\ &\text{with } \delta(i) = 1 \text{ if } i=l, \text{ and } \delta(i) = 0 \text{ if } i \neq l.
\end{align*}
\vspace{0.1cm}
\begin{align*}
        \W_{5_{j}} &= \begin{bmatrix}
    \I & \undermat{\dmj \times \vphantom{\dmj}}{0 & \dotsb & 0} & \undermat{~\left(\dMj - \dmj\right) \times \vphantom{\dmj}}{~~~0 & \dotsb & 0} & \undermat{~\left(d_M - \dMj\right) \times \vphantom{\dmj}}{~~~0 & \dotsb & 0}~ \\ 
    0 & \I & \dotsb & \I & ~~~0 & \dotsb & 0 &~~~ 0 & \dotsb & 0~ \\
    0 & 0 & \dotsb & 0 & ~~~ \I & \dotsb & \I & ~~~0 & \dotsb & 0 \\
    & & & & & \W_{a_{j}} & & & &
    \end{bmatrix}, \\
    \W_{a_{j}} &= 
    \begin{bmatrix}
     0_{(2-j)n\times d_{M_j}n} & \I_{n(2-j)} & \dotsb & \I_{n(2-j)} 
    \end{bmatrix},
\end{align*}
\vspace{0.1cm}
\begin{align*}
    \mathcal{P}_{_j} = \text{diag}\left({\mathcal{H}}_{_j}, \undermat{~\left(d_M - \dMj\right) \times}{\Qthree, \dotsb, \Qthree} \right)+(2-j)\mathcal{H}_a{_j},
\end{align*}
\begin{align*}
{\mathcal{H}}_{_j} \hspace{-0.05cm}=\hspace{-0.05cm}
\begin{bmatrix}
\vspace{0.2cm}  \mathcal{P}_0 \hc \mathcal{P}_{b_{_1}} \hc 0 \hc \dotsb \hc 0 \hc \dotsb \hc 0 \\ \vspace{0.2cm}
  \star \hc \mathcal{P}_{a_{_1}}  \hc \ddots \hc \ddots \hc \vdots \hc \ddots \hc \vdots \\ \vspace{0.2cm}
  \star \hc \ddots \hc \ddots \hc \mathcal{P}_{_{b_{\dmj}}} \hc 0 \hc \dotsb \hc 0 \\ \vspace{0.2cm}
  \vdots \hc \ddots \hc \star \hc \mathcal{P}_{_{a_{\dmj}}} \hc \mathcal{P}_{_{d_{1}}} \hc \ddots \hc \vdots \\ \vspace{0.2cm}
  \star \hc \dotsb \hc \star \hc \star \hc \mathcal{P}_{c_{_1}} \hc \ddots \hc 0  \\ \vspace{0.2cm}
  \vdots \hc \ddots \hc \vdots \hc \ddots \hc \ddots \hc \ddots \hc \mathcal{P}_{d_{_{\hdj}}} \\ \vspace{0cm}
  \star \hc \dotsb \hc \star \hc \dotsb \hc \star \hc \star \hc \mathcal{P}_{c_{_{\hdj}}}
\end{bmatrix}
\end{align*}
\begin{align*}
    &\mathcal{P}_0= {\Z}_{1_{j}} \dmj^{2}+{\Z}_{2_{j}}\hdj^2, \\
    &\mathcal{P}_{a_{i}}={\Q}_{1_{j}}+2{\Z}_{2_{j}}\hdj^2+{\Z}_{1_{j}}\dmj\left(2\dmj-2i+1\right), \\
    &\mathcal{P}_{b_{i}}=-{\Z}_{2_{j}}\hdj^2-{\Z}_{1_{j}}\dmj\left(\dmj-i+1\right), \\
    &\mathcal{P}_{c_{l}}={\Q}_{2_{j}}+{\Z}_{2_{j}}\hdj\left(2\hdj-2l+1\right), \\
    &\mathcal{P}_{d_{l}}=-{\Z}_{2_{j}}\hdj\left(\hdj-l+1\right),
\end{align*}
\noindent for $i \in [1,\dmj]$ and $l \in [1,\hdj]$. Moreover, $\mathcal{H}_a{_j}$ is a matrix of same format of $\mathcal{H}_{_j}$ but with the following definitions:
\begin{align*}
    &\mathcal{P}_0= \Zthree d_{\Delta_M}^2, \mathcal{P}_{a_{i}}=2\Zthree d_{\Delta_M}^2, \mathcal{P}_{b_{i}}=-\Zthree d_{\Delta_M}^2,\\
    &\mathcal{P}_{c_{l}}=\Zthree d_{\Delta_M}\left(2d_{\Delta_M}-2l+1\right), \\
    &\mathcal{P}_{d_{l}}=-\Zthree d_{\Delta_M}\left(d_{\Delta_M}-l+1\right),
\end{align*}
 \noindent for $i \in [1,\dMj]$ and $l \in [1,d_{\Delta_M}]$, where $d_{\Delta_M} = d_M-d_{{M_j}}$.

\end{document}